\documentclass[11pt,a4paper]{article}
\usepackage[latin1]{inputenc} 
\usepackage[english]{babel}
\usepackage[T1]{fontenc} 
\usepackage{amsfonts,amssymb,latexsym,epsfig,amsmath}
\usepackage{graphicx}
\usepackage{amsthm}
\usepackage{dsfont}
\usepackage{color}
\usepackage{hyperref}

\newtheorem{theorem}{Theorem}[section]
\newtheorem{proposition}[theorem]{Proposition}

\newtheorem{lemma}[theorem]{Lemma} 
\theoremstyle{remark}

\numberwithin{equation}{section}
\newcommand{\N}{\hbox{$ I\kern -0.23em N$}}

\newcommand{\R}{{\mathbb R}}

\begin{document}
\title
{Global diffeomorphism of the Lagrangian flow-map defining Equatorially trapped water waves}

\author{Silvia Sastre-G\'{o}mez  %
  \thanks{School of Mathematical Sciences, University College Cork, Ireland. Supported by the Science Foundation Ireland (SFI) under the grant 13/CDA/2117. email: silvia.sastregomez@ucc.ie.}}

\date{}

\maketitle

\begin{abstract}
The aim of this paper is to prove that a three dimensional  
Lagrangian flow which defines equatorially trapped water waves is 
dynamically possible. This is achieved by applying a mixture of analytical 
and topological methods to prove that the nonlinear exact solution to the 
geophysical governing equations,   
derived by Constantin in \cite{Con2012}, is a global diffeomorphism from 
the Lagrangian labelling variables to the fluid domain beneath the free surface. 
\end{abstract}

\section{Introduction}
In this paper we apply a mixture of analytical and topological methods to 
establish  that a recently derived solution   defining Equatorially trapped 
waves is dynamically possible. This remarkable solution, derived by 
Constantin in \cite{Con2012} and given below by equation \eqref{exact_sol}, 
is an exact solution of the nonlinear $\beta-$plane governing equations for 
Equatorial water waves, and it is explicit in the Lagrangian framework. The 
main result of this paper establishes that the three-dimensional mapping 
\eqref{exact_sol} from the Lagrangian labelling domain to the fluid domain 
defines a global diffeomorphism--- a consequence of which is that the 
solution \eqref{exact_sol} defines a fluid motion which is dynamically 
possible.  We achieve this result by first establishing that \eqref{exact_sol} is 
locally diffeomorphic and injective, and then we render our results global by 
applying a suitable version of the classical degree-theoretic {\em Invariance 
of Domain} Theorem, cf. \cite{Deimling,Wieslaw}.

The solution presented by Constantin in \cite{Con2012} represents a 
geophysical generalization of the celebrated Gerstner's wave, in the sense 
that ignoring Coriolis terms in \eqref{exact_sol} recovers Gerstner's wave 
solution. The primary importance of Gerstner's wave is probably the fact that 
it represents the only known explicit and exact solution of the nonlinear 
periodic gravity wave problem with a non-flat free-surface. Gerstner's wave 
is a two-dimensional wave propagating over a fluid domain of infinite depth 
(cf. \cite{Constantin_2001,Constantin_book,Henry_2008}), and interestingly 
it may be modified to describe edge-waves propagating over a sloping bed 
\cite{Constantin_edge_2001,Stuhl}. 
The geophysical solution presented in \cite{Con2012} encompasses 
Gerstner's solution, yet it also possesses a number of inherent 
characteristics which transcends Gerstner's wave. The solution \eqref
{exact_sol} is a truly three-dimensional eastward-propagating geophysical 
wave, and furthermore it is Equatorially-trapped--- achieving its greatest 
amplitude at the Equator and exhibiting a strong exponential decay in 
meridional directions away from the Equator. The solution is furthermore 
nonlinear, as is seen from the wave-surface profile, and has a dispersion 
relation that is dependant on the Coriolis parameter. 

Since the solution \eqref{exact_sol} is explicit in the Lagrangian formulation, 
we may immediately discern some qualitative properties of the physical fluid 
motion. Indeed, an advantage of solutions in the Lagrangian framework is 
that the fluid kinematics may be explicitly described \cite{Bennett}. From 
\eqref{exact_sol} we see that at each fixed latitude the solution  prescribes 
individual fluid particles to move clockwise in a vertical plane. Each particle 
moves in a circle, with the diameter of the circles decreasing exponentially 
with depth. 
In \cite{Con2012} it was simply shown that the solution \eqref{exact_sol} is 
compatible with the  governing equations of the $\beta-$plane approximation 
for Equatorial water waves \eqref{incompress_eq}-\eqref{asymp_beh}. 

The aim of this paper is to rigorously justify that the fluid motion defined by 
\eqref{exact_sol} is dynamically possible.  This is achieved by establishing 
that the solution \eqref{exact_sol} defines a global diffeomorphism, thereby 
ensuring that it is indeed possible to have a three-dimensional motion of the 
whole fluid body where all the particles describe circles with a depth-
dependant radius at fixed latitudes, and furthermore the particles never 
collide but instead they fill out the entire infinite region below the surface 
wave. In so doing we show that the fluid domain as a whole evolves in a 
manner which is consistent with the full governing equations.  We note that 
subsequent to the derivation of Constantin's solution, a wide range of 
geophysical generalizations and variations to \cite{Con2012} have been 
produced and analysed, for example 
\cite{ConJPO, R1, ConstGer, GenHen2014, Hen2013, HenHsu2015a,  
HenHsu2015b, I-K, Mat2012, Mat2013}. It is expected that the rigorous 
considerations of this paper are also applicable to these variants.

\section{The Equatorially trapped wave solution}
\subsection{Governing equations}
We consider geophysical waves in the Equatorial region, where we 
assume that 
the earth is a perfect sphere of radius $R=6378$ km. We are in a rotating 
framework, where the $x$-axis is  facing horizontal due east (zonal 
direction), the $y$-axis is due 
north (meridional direction), and the $z$-axis is pointing vertically upwards. 
The governing equations for geophysical ocean waves are given by  Euler's 
equation with additional terms involving the Coriolis parameter which is 
proportional to the rotation 
speed of the earth, see \cite{Cushman_Roisin, Pedlosky} 
\begin{equation}\label{Euler_eq}
	\left\{
	\begin{array}{ll}
		u_t+uu_x+vu_y+wu_z\,+2\Omega w \cos \Phi
		-2\Omega v\sin\Phi 
		& \displaystyle= -{1\over \rho} P_x\smallskip\\
		v_t\,+uv_x\,+vv_y\,+wv_z\,+2\Omega u \sin\Phi 
		& \displaystyle = -{1\over \rho} P_y\smallskip\\
		w_t+uw_x+vw_y+ww_z-2\Omega u \cos\Phi 
		& \displaystyle = -{1\over \rho} P_z-g,
	\end{array}
	\right.
\end{equation}
the mass conservation equation
\begin{equation}\label{mass_con_eq}
	\rho_t+u\rho_x+v\rho_y+w\rho_z=0	
\end{equation}
and the equation of incompressibility
\begin{equation}\label{incompress_eq}
	u_x+v_y+w_z=0. 
\end{equation}
Here $\Phi$ represents the latitude, $(u,v,w)$ is the fluid velocity,   
$\Omega=73.10^{-6}$~rad$/$s is the (constant) rotational speed of earth 
(which is the sum of the rotation of the earth about its axis and the rotation 
around the sun, see \cite{Cushman_Roisin}), $g=9.8$~m/s$^{-2}$ is the 
gravitational constant, $\rho$ is the water density, and $P$ is the pressure. 

We are interested in Equatorial waves, that is, geophysical ocean waves 
in a region 
which is within $5^o$ latitude of the Equator. Since the latitude is small, 
we may use the approximations $\sin \Phi\approx \Phi$, and $\cos\Phi
\approx 1$, and thus linearising the Coriolis force leads to the $\beta$-plane 
approximation to equations \eqref{Euler_eq} given by
\begin{equation}\label{beta_plane_eq}
	\left\{
	\begin{array}{ll}
		u_t+uu_x+vu_y+wu_z+2\Omega w
		-\beta y v 
		& \displaystyle= -{1\over \rho} P_x\smallskip\\
		v_t+\,uv_x\,+vv_y+\,wv_z+\beta y u
		& \displaystyle = -{1\over \rho} P_y\smallskip\\
		w_t+uw_x+vw_y+ww_z-2\Omega u 
		& \displaystyle = -{1\over \rho} P_z-g,
	\end{array}
	\right.
\end{equation}
where $\beta=2\Omega/R=2.28\cdot 10^{-11}$~m$^{-1}$s$^{-1}$. 
The relevant boundary conditions are the kinematic boundary conditions 
\begin{equation}\label{kinematic_boundary_cond_1}
	w=\eta_t+u\eta_x+v\eta_y \mbox{ on } z=\eta(x,y,t),
\end{equation}
\begin{equation}\label{kinematic_boundary_pressure}
	P=P_{atm} \mbox{ on } z=\eta(x,y,t),
\end{equation}
where $P_{atm}$ is the (constant) atmospheric pressure, and $\eta(x,y,t)$ 
is the free surface. The boundary condition \eqref
{kinematic_boundary_cond_1} states that all the particles in the surface 
will stay in the surface for all time $t$, and the boundary condition \eqref
{kinematic_boundary_pressure} decouples the water flow from the motion 
of the air above. We work with an infinitely-deep fluid domain and so we 
require  
the velocity field to converge rapidly to zero with depth, that is
\begin{equation}\label{asymp_beh}
(u,w)\to (0,0)\;\mbox{ as }\; z\to-\infty. 
\end{equation}
The governing equations for the $\beta-$plane approximation of geophysical 
ocean waves are given by 
\eqref{incompress_eq}-\eqref{asymp_beh}.

\subsection{Exact solution}
In this section we present and describe briefly the exact solution of the 
$\beta$-plane governing 
equations \eqref{incompress_eq}-\eqref{asymp_beh} which was recently 
derived by Constantin \cite{Con2012}. This solution describes a three-
dimensional eastward-propagating geophysical wave which is Equatorially 
trapped, exhibiting a strong exponential decay in meridional directions away 
from the Equator, and which is periodic in the zonal direction. 
Equatorially trapped waves  propagating eastward and symmetric about the 
Equator are known to exist,  and they are regarded as one of the key factors 
in a possible explanation of  the El Ni\~no phenomenon 
(cf. \cite{Cushman_Roisin, Fedorov_Brown}). 
The formulation of the solution employs a Lagrangian viewpoint, describing 
the evolution in time of an individual fluid particle \cite{Bennett}. 
The Lagrangian positions of the fluid $(x,y,z)$ are given in terms of the 
labelling 
variables $(q,r,s)$, and time $t$ by 
\begin{equation}\label{exact_sol}
	\left\{
	\begin{array}{rl}
		x&\displaystyle=q-{1\over k} e^{k\left[r-f(s)\right]}\sin \left[k(q-ct)
		\right],\\
		y&\displaystyle=s,\\
		z&\displaystyle=r+{1\over k}e^{k\left[r-f(s)\right]}\cos \left[k(q-ct)
		\right],
	\end{array}
	\right.
\end{equation}
where $k$ is the wave number, defined by $k=2\pi/L$ where $L$ is the 
wavelength, and the wave phase speed is determined by the dispersion 
relation
%
\begin{equation}\label{disp_eq}
	c={\sqrt{\Omega^2+kg}-\Omega\over k},
\end{equation}
and also
\begin{equation}\label{f_eq}
	f(s)={c\beta\over 2g}s^2
\end{equation}
determines the decay of fluid particle oscillations in the meridional direction. 
The labelling 
variables take the values $(q,r,s)\in \R\times (-\infty,r_0)\times\R$, where 
 $r_0\leq 0$ is fixed.  
For every fixed $s$, the system \eqref{exact_sol} describes the flow 
beneath a surface wave propagating eastwards (in the $x$-direction) at constant 
speed $c$ determined by \eqref{disp_eq}.  
At fixed latitudes (that is, for $s$ fixed) the free surface $z=\eta(x,y,t)$ is obtained by setting $r=r_0(s)$ in the third equation in \eqref{exact_sol}, where $r_0(s)<r_0$ is the unique solution to
\[
	{e^{2k\left[r_0(s)-f(s)\right]}\over 2k}-r_0(s)
	={e^{2kr_0}\over 2k}-r_0.
\]
A plot of the free-surface for the wave solution \eqref{exact_sol} is given in Figure \ref{Gerstner} below.
\begin{figure} [htp]
\begin {center}
\includegraphics [height=7.5cm]{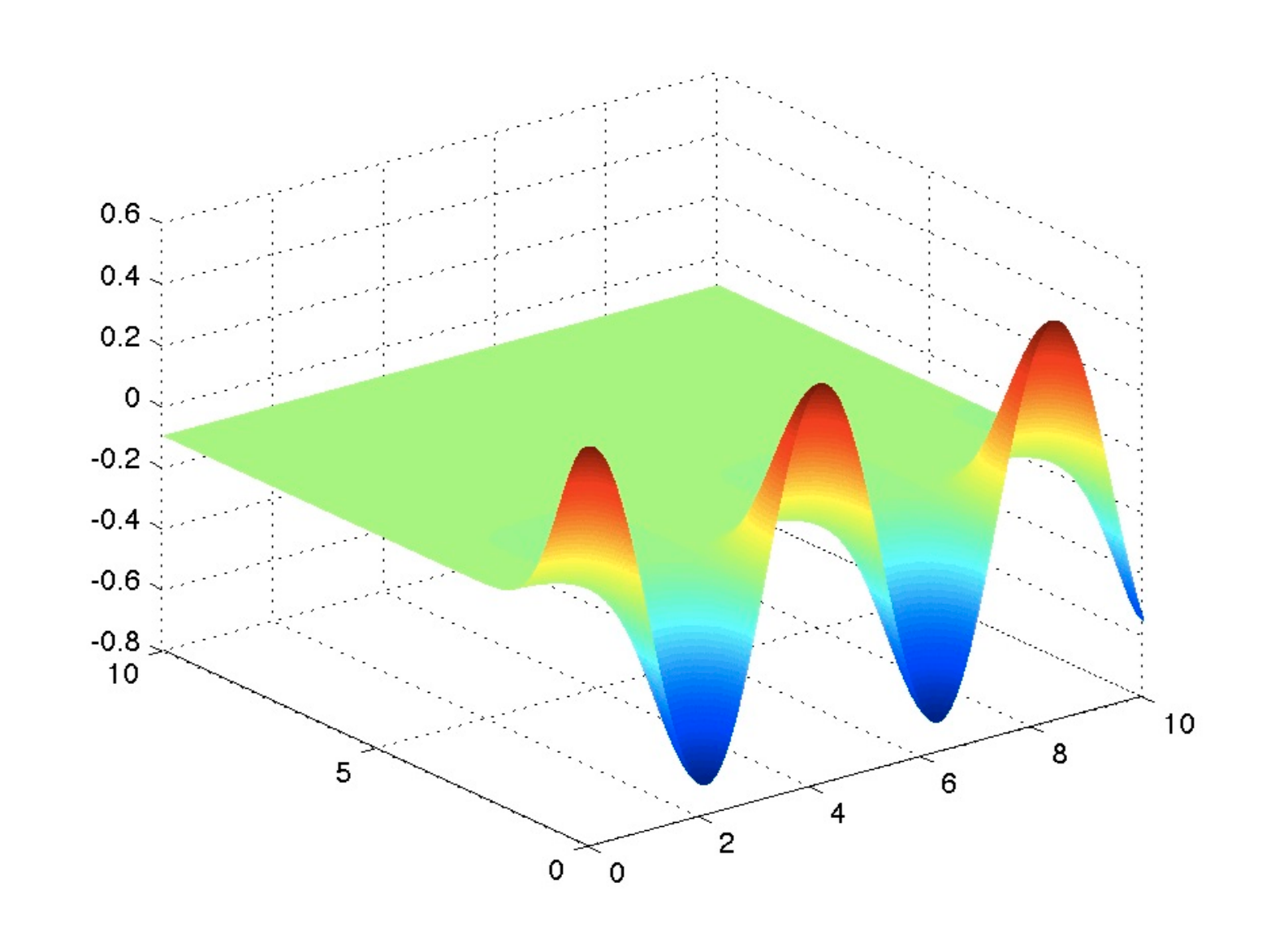}
\caption{Exact solution  \eqref{exact_sol}.}\label{Gerstner}
\end {center}
\end{figure}

In \cite{Con2012} the author focuses on proving, by explicit computation, that 
the exact solution 
\eqref{exact_sol} is compatible with the governing equations  
\eqref{incompress_eq}-\eqref{asymp_beh}. Our aim in this work is to prove that it is 
dynamically possible to have a global motion of the fluid domain where, at fixed latitudes, the particles move in circular paths 
with depth-dependant radius. Indeed, we prove in our main result 
Proposition~\ref{prop} that the fluid motion defined by \eqref{exact_sol} 
is dynamically possible, that is, at any instant $t$, the label map 
 is a global diffeomorphism from the labelling variables,  $\{(q,r,s): q\in\R,\, r
 \le r_0\mbox{ and } s\in\R\}$,  to the fluid domain beneath the free surface 
 given by 
%
\begin{equation}\label{surface_profile}
	(q,s)\mapsto \left(q-{1\over k}e^{r_0(s)-f(s)}\sin\left[k(q-ct)\right], s, r_0(s)+{1\over 
	k}e^{r_0(s)-f(s)}\cos\left[k(q-ct)\right]\right).
\end{equation}
For a fixed latitude $s$, the surface wave profile \eqref{surface_profile} is 
a reverse trochoid if 
$r_0(s)<0$ and a reverse cycloid with a cusp at the wave crest if $r_0(s)=0$ and 
$s=0$.  

%
Fixed $s$, and given $k>0$ and $r_0(s)\le 0$,  the curve $z=h_s(x)$ given parametrically by 
\begin{equation}\label{trochoid}
	\xi\mapsto \left({\xi\over k}-{e^{r_0(s)-f(s)}\over k}\sin\left(\xi\right), {1\over 
	k}-{e^{r_0(s)-f(s)}\over k}\cos\left(\xi\right)\right)
\end{equation}
is a trochoid if $r_0(s)-f(s)<0$ and a cycloid if $r_0(s)-f(s)=0$. It represents the 
curve traced by a fixed point at a distance ${e^{r_0(s)-f(s)}\over k}<{1\over k}$ 
from the center of a circle of radius ${1\over k}$ rolling along a straight line 
without slipping, (see Figure \ref{cycloid_trochoid}).

\begin{figure} [htp]
\begin {center}
\includegraphics [height=6cm]{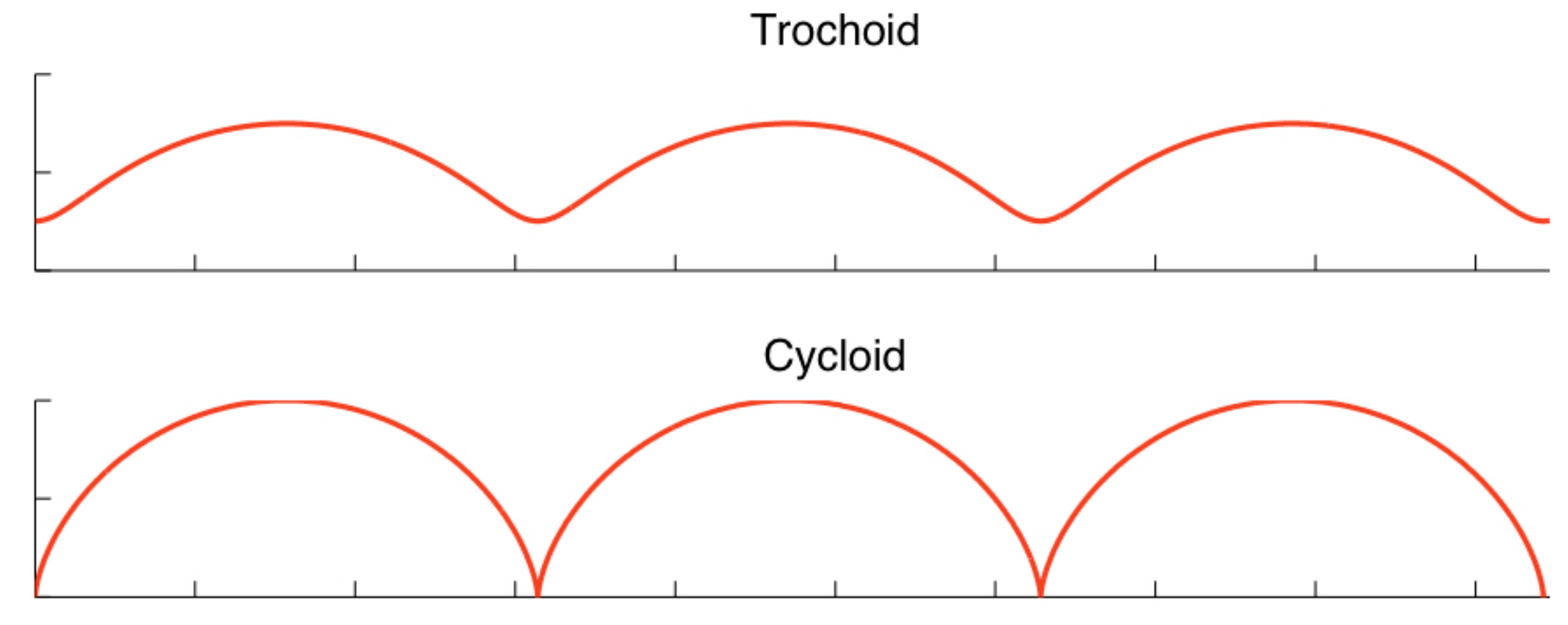}
\caption{Trochoid and cycloid curves.}\label
{cycloid_trochoid}
\end {center}
\end{figure}
Therefore, for a fixed latitude $s$, the free surface of the fluid has the 
equation $z=r_0(s)+{1\over k}-h_s
(x-ct)$ which represents a reverse trochoid propagating to the right with 
velocity $c$. Since $h_s$ is periodic with minimal period ${2\pi/k}$ then the  
surface is a periodic wave with period ${2\pi/k}$, which concurs with the 
definition of the wave number $k$.

\section{Main results}
To prove that the motion \eqref{exact_sol} is dynamically possible, it is 
sufficient to analyse \eqref{exact_sol} for the time $t=0$, when it takes the 
form
\begin{equation}\label{exact_sol_t_0}
	\left\{
	\begin{array}{rl}
		x&\displaystyle=q-{e^{k\left[r-f(s)\right]}\over k} \sin \left(kq\right),\\
		y&=s\\
		z&\displaystyle=r+{e^{k\left[r-f(s)\right]}\over k}\cos \left(kq\right).
	\end{array}
	\right.
\end{equation}
 The case of a general time $t$ in  \eqref{exact_sol} is recovered making 
 first the change of variables 
 $(q,r,s)\mapsto (q+ct,r,s)$, performing \eqref{exact_sol_t_0}, and finally 
 shifting 
 the horizontal variable $x$ by $ct$. 
 Therefore we can focus on \eqref{exact_sol_t_0}, and we further note that 
 as $q$ varies by $2\pi/k$, the $z$ 
 value reoccurs and $x$ is shifted linearly by $2\pi/k$. Hence, it suffices 
 to analyse  \eqref{exact_sol_t_0} on the domain
 $$\mathcal{D}=\left\{(q,r,s): q\in\left[0,
 {2\pi\over k}\right], r \le r_0 \mbox{ and } s\in\R\right\}.$$
In the following result we first prove that the map \eqref{exact_sol_t_0} is an injective local 
diffeomorphism.
\begin{lemma}\label{inject}
		For every fixed $t\ge 0$, if $r_0<0$ then the map \eqref{exact_sol_t_0} is a 
		local diffeomorphism from $\mathcal{D}=\left\{(q,r,s): q\in\left[0,
		{2\pi\over k}\right], r \le r_0 \mbox{ and } s\in\R\right\}$ into its image injectively. In the limiting case $r_0=0$, then this result holds for the domain $\mathcal D$ excluding the cusps at the equatorial wave-crests.
\end{lemma}
\begin{proof}
We remark first that $r-f(s)\leq r_0-f(s)\leq r_0\leq 0$, as we see from the definition of $r_0$ and \eqref{f_eq}.  The differential of 
\eqref{exact_sol_t_0} at a point $(q,r,s)$ is given by 
\begin{equation}
	\left(
	\begin{array}{ccc}
	 1-e^{k(r-f(s))}\cos(kq) &   f'(s)e^{k(r-f(s))}\sin(kq) &-e^{k(r-f(s))}\sin(kq) 
	 \\
	 0 & 1 & 0
	 \\
	  -e^{k(r-f(s))}\sin(kq)& f'(s)e^{k(r-f(s))}\cos(kq) &  1+e^{k(r-f(s))}\cos(kq)   
	\end{array}
	\right)
\end{equation}
with determinant $1-e^{2k(r-f(s))}$. As an aside, we note that the time independence of this expression implies that the fluid is incompressible and so \eqref{incompress_eq} holds, cf. \cite{Con2012}.  It follows that if $r_0<0$ the Jacobian of \eqref{exact_sol_t_0} is non-zero (strictly positive) everywhere, whereas in the case $r_0=0$ the Jacobian is zero precisely at the Equator ($s=0$), where the break-down in regularity corresponds to the appearance of cusps at the wave-crest as discussed above. Therefore, aside from the situation when $r_0=s=0$, the mapping \eqref{exact_sol_t_0} is differentiable, continuous with  non-zero derivative, and hence we can apply the Inverse Function Theorem to infer that  
\eqref{exact_sol_t_0} is a smooth local diffeomorphism onto its image.

Let us prove now that \eqref{exact_sol_t_0} is injective. Let 
$(q_i,r_i,s_i)\in\mathcal{D}$ for $i=1,2$,  and let $\big(x(q_i,r_i,s_i), y
(q_i,r_i,s_i), z(q_i,r_i,s_i)\big)$ be the corresponding fluid particles given by 
\eqref{exact_sol_t_0}. 
First of all,  if 
$$\big(x(q_1,r_1,s_1),y(q_1,r_1,s_1),z(q_1,r_1,s_1)\big)=\big(x(q_2,r_2,s_2),y
(q_2,r_2,s_2),z(q_2,r_2,s_2)\big)$$
then $s_1=s_2$. Thus, we can fix $s$ and then focus on checking injectivity 
with respect to  $x$ and $z$ in \eqref{exact_sol_t_0}. 
Letting $\xi=q+ir $, 
then the values of $(x,z)$ in \eqref{exact_sol_t_0} correspond to the map 
\[
	\xi\mapsto \xi+i{e^{-kf(s)}\over k}e^{ik\overline\xi}.
\]
To prove injectivity, we consider $F(\xi)=\xi+g(\xi)$, where $\xi=(q,r)$. Let 
$\xi_1\ne\xi_2$, then applying the Mean Value Theorem we derive 
\begin{equation}\label{ineq_inject}
	\begin{array}{ll}
	\left|F(\xi_1)-F(\xi_2)\right|&\displaystyle \ge \left|\xi_1-\xi_2\right|-\left|g
	(\xi_1)-g	(\xi_2)\right| \smallskip\\
	&\displaystyle\ge    \left|\xi_1-\xi_2\right|-\max\limits_{s\in[0,1]}\|Dg_{s\xi
	+(1-s)\xi_2}\|\left|\xi_1-\xi_2\right|.
	\end{array}
\end{equation}
Computing $Dg$ in terms of $(q,r)$, yields 
\begin{equation}
	Dg_{(q,r)}=e^{k(r-f(s))}\left(
	\begin{array}{cc}
	 -\cos(kq) &    -\sin(kq) \\
	  -\sin(kq)&   \cos(kq)   
	\end{array}
	\right)
\end{equation}
then  $\|Dg_{(q,r)}\|=e^{2k(r-f(s))}$. From \eqref{ineq_inject}, and considering that 
$\xi_i=(q_i,r_i)$ for $i=1,2$,  we obtain that 
\begin{equation*}
	\begin{array}{ll}
	\left|F(\xi_1)-F(\xi_2)\right|
	&\displaystyle\ge \displaystyle   \left|\xi_1-\xi_2\right|-e^{2k(\max
	\{r_1,r_2\}-f(s))}\left|
	\xi_1-\xi_2\right|\\
	&\displaystyle=  \big(1-e^{2k(\tilde{r}-f(s))}\big)\left|\xi_1-\xi_2\right|,
	\end{array}
\end{equation*}
where $\tilde{r}=\max\{r_1,r_2\}$.  Therefore, if $\tilde{r}-f(s)<0$, $F$ is 
injective, and we have proved that \eqref{exact_sol_t_0} is injective.
\end{proof}
The following result  will be used to prove that 
\eqref{exact_sol} is in fact a global diffeomorphism, cf. \cite{Wieslaw}, \cite{Rothe}.
\begin{theorem}\label{Invariance_domain} {\bf (Invariance of Domain 
theorem)}  
	If $U\subset \R^n$ is open and $F:U\to \R^n$ is a continuous one-to-one 
	mapping, then $F:U\to F(U)$ is a homeomorphism, and 
	$F(\partial\overline{U})=\partial F(\overline{U})$.
\end{theorem}

We have already proved in Lemma \ref{inject} that the exact solution \eqref{exact_sol_t_0} 
gives us a local diffeomorphism that is globally injective on $\mathcal{D}$. The result below proves  
that \eqref{exact_sol} is a global diffeomorphism for all $t\ge 0$, this is, that \eqref{exact_sol} 
is dynamically possible.

\begin{proposition}\label{prop}
	For every fixed $t\ge 0$, if $r_0\le 0$ the map \eqref{exact_sol} is a global 
	diffeomorphism 
	from $\mathcal{V}=\left\{(q,r,s): q\in\R, r < r_0 \mbox{ and } s\in\R\right\}$ 
	into 
	the fluid domain beneath the free surface $z=\eta(x,y,t)$. Moreover, if 
	$r_0<0$ the free surface $z=\eta(x,y,t)$ has a smooth profile, and in the 
	limiting 
	case $r_0=0$ the free surface is 
	piecewise smooth with upward cusps at $s=0$.
\end{proposition}
\begin{proof}	
From Lemma \ref{inject}, we know that the map \eqref{exact_sol_t_0} 
is an injective local diffeomorphism from $\mathcal{D}=\left\{(q,r,s): q\in\left[0,
{2\pi\over k}\right], r \le r_0 \mbox{ and } s\in\R\right\}$ into its image.
To prove that the local diffeomorphism is in fact a global diffeomorphism we 
just have  to prove that it is a homeomorphism. Indeed, since the 
hypotheses in the Invariance of Domain Theorem \ref{Invariance_domain} 
are satisfied, then the map \eqref{exact_sol_t_0} is a homeomorphism. 
Although it is guaranteed by the Invariance Domain Theorem \ref{Invariance_domain}, 
we can see  directly that   
the map \eqref{exact_sol_t_0} sends $\partial \overline{\mathcal{D}}$ 
into the boundaries of the image of $\overline{\mathcal{D}}$.  
The vertical semiplanes $\left\{(0,r,s): r \le r_0\mbox{ and } s\in\R\right\}$ and 
$\left\{\!(2\pi /k,r,s): r \le 
r_0\mbox{ and } s\in\R\right\}$ are transformed by \eqref{exact_sol_t_0} in the  
 vertical  surfaces  \linebreak
$\left\{\!(0,y,z)\!:\!  z \le r_0(s)+{e^{k(r_0(s)-f(s))}\over k},\, y\in\R\right\}$ and 
$\left\{(2\pi /k,y,z)\!:\!z \le r_0(s)+{e^{k(r_0(s)-f(s))}\over k},\, y\in\R\right\}
$ respectively, and 
the horizontal semiplane  $\left\{(q,r_0,s): 0\le q\le 2\pi/k \mbox{ and } 
s\in\R\right\}$ 
becomes part of the reverse trochoid if $r_0(s)-f(s)<0$, which is smooth, and it 
becomes part of the reverse cycloid if $r_0(s)=0$ and $s=0$, which is piecewise 
smooth with upward cusps.

We have proved 
that \eqref{exact_sol_t_0} is a global diffeomorphism map from $\mathcal{D}$ 
into its image if $r_0<0$, with singularities occurring when $r_0=0$ and 
$s=0$. 
Since the full system \eqref{exact_sol} can be recovered from  \eqref{exact_sol_t_0} 
by making the change of variables $(q,r,s)\mapsto (q+ct,r,s)$, and finally shifting 
the horizontal variable $x$ by $ct$, it follows that \eqref{exact_sol} is a global 
diffeomorphism from  $\mathcal{V}=\left\{(q,r,s): q\in\R, r < r_0 \mbox{ and } 
s\in\R\right\}$ into the fluid domain below the free surface. 
\end{proof}



\end{document}